\documentclass[10pt]{article}
\pdfoutput=1
\usepackage[T1]{fontenc}
\usepackage[utf8]{inputenc}
\usepackage[a4paper,margin=1in]{geometry}
\usepackage{amsmath,amssymb,amsthm,mathtools}
\usepackage{microtype}
\usepackage{authblk}
\usepackage[colorlinks=true,linkcolor=blue,citecolor=blue,urlcolor=blue]{hyperref}
\hypersetup{pdftitle={On a hypergraph Tur\'an problem of Balogh--Bohman--Bollob\'as--Zhao}}

\theoremstyle{plain}
\newtheorem{theorem}{Theorem}[section]
\newtheorem{lemma}[theorem]{Lemma}
\newtheorem{proposition}[theorem]{Proposition}
\newtheorem{corollary}[theorem]{Corollary}
\theoremstyle{definition}
\newtheorem{definition}[theorem]{Definition}
\theoremstyle{remark}

\newcommand{\Prb}{\mathbb{P}}

\newcommand{\oo}{o(1)}

\title{\Large \bf On a hypergraph Tur\'an problem of Balogh--Bohman--Bollob\'as--Zhao}

\date{\today}
\begin{document}

\author[1,2]{Heng~Li}
\author[3,4]{Jie~Ma}
\author[3]{Tianhen~Wang}
\author[5]{Yixiao~Zhang}
\author[3]{Tianming~Zhu}
\affil[1]{\small School of Mathematics, Shandong University, Jinan, China}
\affil[2]{\small Extremal Combinatorics and Probability Group, Institute for Basic Science, Daejeon, South
Korea}
\affil[3]{{\small School of Mathematical Sciences, University of Science and Technology of China, Hefei, China}}
\affil[4]{{\small Yau Mathematical Sciences Center, Tsinghua University, Beijing, China}}
\affil[5]{\small Center for Discrete Mathematics, Fuzhou University, Fujian, China}


\maketitle

\begingroup
\renewcommand\thefootnote{}
\footnotetext{Emails: \texttt{heng.li@sdu.edu.cn}; \texttt{jiema@ustc.edu.cn}; \texttt{wth1115060377@mail.ustc.edu.cn}; \texttt{fzuzyx@gmail.com}; \texttt{zhutianming@mail.ustc.edu.cn}.}
\endgroup

\begin{abstract}
Let $S$ and $T$ be disjoint sets with $|S|=i$ and $|T|=r-1$ for $2\le i\le r-1$, and let $B_i^{(r)}$ be the $r$-graph on $S\cup T$ whose edges are the $r$-subsets containing $S$ or $T$.
We study the deficit $q_{r,i}:=1-\pi(B_i^{(r)})$ in its Tur\'an density.
Balogh, Bohman, Bollob\'as, and Zhao previously obtained bounds for these deficits with logarithmic gaps near both ends of the sequence $B_i^{(r)}$, namely, when $i=O(1)$ or $i=r-O(1)$.
We close these gaps by showing that, as $r\to\infty$, for every fixed integer $a\ge1$, $q_{r,a+1}=\Theta_a(r^{-a})$, and for every fixed integer $b\ge2$, $q_{r,r-b}=\Theta_b(r^{-b}\log r)$.
\end{abstract}

\section{Introduction}

For a set $V$ and an integer $r\ge 1$, denote by $\binom{V}{r}$ the family of all $r$-element subsets of $V$.  
An \emph{$r$-uniform hypergraph} (or $r$-graph) $G$ consists of a vertex set $V(G)$ and an edge set $E(G)\subseteq\binom{V(G)}{r}$.  
We write $|G| = |E(G)|$ for the number of edges.
For an $r$-graph $H$, the \emph{Tur\'{a}n number} $\mathrm{ex}(n,H)$ is the maximum number of edges in an $H$-free $r$-graph on $n$ vertices.  
The \emph{Tur\'{a}n density} $\pi(H)$ is defined as  
\[
\pi(H)\coloneqq \lim_{n\to\infty} \mathrm{ex}(n,H) / \tbinom{n}{r},
\]
and the existence of this limit is guaranteed by a result of Katona, Nemetz, and Simonovits~\cite{KatonaNemetzSimonovits64}. 

In the setting of graphs (i.e., $r=2$), the asymptotic behavior of Tur\'{a}n numbers is well understood due to the celebrated Erd\H{o}s--Stone--Simonovits theorem, which determines $\pi(H)$ in terms of the chromatic number $\chi(H)$ of the graph $H$. To be precise, they proved that $\pi(H)=1-1/(\chi(H)-1)$ for any graph $H$ with $\chi(H)\ge 3$. For hypergraphs of uniformity $r\ge 3$, however, the situation is much more complex. No simple parameter analogous to the chromatic number is known to determine \(\pi(H)\). Indeed, even the Tur\'{a}n density of the complete $3$-graph $K_4^{(3)}$ (the tetrahedron) is not known exactly. The best known bounds are $5/9 \le \pi(K_4^{(3)}) \le 0.5616$, where the lower bound is due to Tur\'{a}n~\cite{Turan41} (see also~\cite{DeCaen91}), and the upper bound was established by Razborov~\cite{Razborov10} using his flag algebra method.

For general $r$, the problem of determining the Tur\'{a}n density of $K_{r+1}^{(r)}$, the complete $r$-graph on $r+1$ vertices, is even less understood. For $r=4$, Markstr\"{o}m~\cite{Markstrom} proved that $\pi(K_5^{(4)}) \le 1753/2380 \approx 0.73655$, while a construction of Giraud~\cite{Giraud} gave the lower bound $11/16$. Considerable attention has also been devoted to estimating \(\pi(K_{r+1}^{(r)})\) as \(r\to\infty\), a problem closely related to the study of the so-called {\it Tur\'an systems}, studied by Frankl and R\"odl~\cite{FR85}, Sidorenko~\cite{Sid97}, and others. Combining de Caen's classical upper bound~\cite{deCaen} with Pikhurko's recent construction~\cite{Pikhurko} for Tur\'an systems shows that, as \(r\to\infty\), the density \(\pi(K_{r+1}^{(r)})\) is
\(1-\Theta(1/r)\); see also the extension by Liu and Pikhurko~\cite{LiuPikhurko}. 

In this note, we study the Tur\'{a}n problem for a family of hypergraphs related to $K_{r+1}^{(r)}$ defined as follows. 
Given integers
$r,s,t$ with $2\le s \le r,~ 2\le t\le r-1$ and $s+t \ge r$, let $F^r_{s,t}$ be the $r$-graph with vertex set $S\cup T$, where
$S\cap T=\emptyset$, $|S|=s$ and $|T|=t$, and with edge set
\[
\left\{e\in\tbinom{S\cup T}{r}: S\subseteq e\text{ or }T\subseteq e \right\}.
\]
Thus $F^r_{2,r-1}=K^r_{r+1}$. 
In this paper, we focus on the deficit
\[
q_{r,i}:=1-\pi(F^r_{i,r-1}) \qquad \mbox{for} \qquad 2\leq i\leq r-1.
\]
Throughout the paper, subscripts on asymptotic notation indicate the parameters on which the implicit constants may depend; asymptotic notation without subscripts has absolute implicit constants.

The study of these hypergraphs was motivated in part by a problem raised by Mubayi (see Problem~25 in the survey~\cite{MubayiProblem}), who asked to determine the growth rate of $1-\pi(F^r_{i,r-1})$ as $r\to\infty$ with $i=i(r)$, and in particular whether
\[
q_{r,r-1}= 1-\pi(F^r_{r-1,r-1}) =\Theta\left(\frac{\log r}{r}\right).
\]

Writing $F^r_{i,r-1}$ as $B_i^{(r)}$, Balogh, Bohman, Bollob\'as, and Zhao~\cite{BBBZ} proved the following theorem. 
\begin{theorem}[\cite{BBBZ}]\label{THM:BBBZ}
Fix \(a,b\ge1\). As \(r\to\infty\), we have
    \[
    \Theta_a(r^{-a})\le q_{r,a+1}\le \Theta_a(r^{-a}\log r)
    \quad\mbox{and}\quad
    \Theta_b(r^{-b})\le q_{r,r-b}\le \Theta_b(r^{-b}\log r).
    \]
\end{theorem}

Thus their theorem leaves
logarithmic gaps near both ends of the sequence $B_2^{(r)},\ldots,B_{r-1}^{(r)}$.  
We refer readers to  the survey of Keevash~\cite{Keevash} for an explanation of why logarithmic deficits are natural near the right end of the sequence.
To put our results in perspective, we briefly survey earlier results related to $q_{r,i}$. First, the case $i=2$ is classical. Since $F^r_{2,r-1}=K^r_{r+1}$, the deficit $q_{r,2}$, as mentioned before, is known to be \(\Theta(1/r)\).  
For general fixed $i$, Balogh, Bohman, Bollob\'as, and Zhao~\cite{BBBZ} also introduced certain modifications of $B_i^{(r)}$ and obtained matching logarithmic estimates for those modified hypergraphs. 
It should be noted that the independent neighborhood problem studied by Bohman, Frieze, Mubayi, and Pikhurko~\cite{BFMP}, corresponding to $\pi(F^r_{r,r-1})$, is distinct from the problem of determining $q_{r,i}$ considered here.

The purpose of
this note is to close the logarithmic gaps of Theorem~\ref{THM:BBBZ} for fixed $i=a+1\ge 2$ at the left end, and for fixed $b=r-i\ge 2$ at the right end (of the sequence $B_i^{(r)}$ for $2\leq i\leq r-1$). Our main results are as follows.

\begin{theorem}\label{thm:small-i}
For every fixed integer $a\ge 1$, as $r\to \infty$ we have 
\[
q_{r,a+1}=O_a(r^{-a}).
\]
\end{theorem}

\begin{theorem}\label{thm:near-endpoint}
For every fixed integer $b\ge2$, set $C = C_{r,b} := \binom{r-1}{b}$. As $r\to \infty$, we have 
\[
q_{r,r-b} \ge (1-\oo)\,\frac{W\left(C/(r-b)\right)}{C},
\]
where \(W\) denotes the principal branch of the Lambert \(W\)-function, that is, the inverse of \(f(x)=xe^x\) on \([0,\infty)\).
\end{theorem}
In particular, for every fixed \(b\ge2\), as \(r\to\infty\),
\[
q_{r,r-b}\ge (1-\oo)\frac{(b-1)\log r}{\tbinom{r-1}{b}} = \Omega_b(r^{-b} \log r).
\]
Together, Theorem~\ref{thm:small-i}, Theorem~\ref{thm:near-endpoint}, and Theorem~\ref{THM:BBBZ} imply the following corollary.

\begin{corollary}
    For all fixed integers $a\ge 1$ and $b \ge 2$, as $r\to \infty$ we have 
    \[
    q_{r,a+1}=\Theta_a(r^{-a}) \quad \text{and} \quad q_{r,r-b}=\Theta_b(r^{-b}\log r).
    \]
\end{corollary}

\noindent This solves two instances of Problem~25 in the survey~\cite{MubayiProblem} due to Mubayi.
Note that Mubayi's problem of determining $q_{r,r-1}$ remains open up to a logarithmic factor. 

The rest of this note is organized as follows.  
In Section~\ref{sec:preliminary}, we introduce notation and the concept of shadow-complete systems.  
In Section~\ref{sec:bound-fixed-i}, we prove Theorem~\ref{thm:small-i} by relating $F_{s,t}^{r}$-free $r$-graphs to shadow-complete systems.  
In Section~\ref{sec:near-endpoint}, we prove Theorem~\ref{thm:near-endpoint} by applying a clique-counting method to obtain the required lower bound on $q_{r,r-b}$ for fixed $b\ge 2$.

\section{Preliminaries}\label{sec:preliminary}

We use the following notation.
For an integer $n\ge 1$, write $[n]=\{1,2,\dots,n\}$. The \emph{complement} \(\overline G\) of an \(r\)-graph \(G\) has vertex set \(V(G)\) and edge set
\(
E(\overline G)=\tbinom{V(G)}{r}\setminus E(G).
\)
For an $r$-graph $G$ and a set $U\subseteq V(G)$, the \emph{induced subgraph} $G[U]$ has vertex set $U$ and edges $\{e\in E(G): e\subseteq U\}$.  
For $0\le s\le r$, the \emph{$s$-shadow} of $G$ is
\[
\partial_s(G) = \left\{ S\in\tbinom{V(G)}{s} : \text{there exists } e\in E(G) \text{ such that } S\subseteq e \right\}.
\]
For an $(r-1)$-set $S$ in an $r$-graph $G$, its \emph{neighborhood} is 
\(
N_G(S)=\left\{x\notin S : S\cup\left\{x\right\}\in G\right\},
\)
and its \emph{degree} is denoted by $d_G(S)=|N_G(S)|.$
For an $(r-b)$-set $S$ in an $r$-graph $G$, the \emph{$b$-uniform link} of $S$ is
\[
L_G(S) := \left\{ A\in\tbinom{V(G)\setminus S}{b} : S\cup A \in G \right\}.
\]

We need the following Chernoff-type bound. For completeness, we present a self-contained proof here.  

\begin{lemma}\label{LEM:Chernoff}
    Let \(X_1, X_2, \dots, X_n\) be independent random variables with \(0 \le X_i \le 1\). Define \(S = \sum_{i=1}^n X_i\) and \(\mu = \mathbb{E}[S]\). Then for any \(0 < \delta < 1\),
    \[
    \Prb\bigl[S \le (1-\delta)\mu\bigr] \le \exp\left(-\mu(1-\delta)\log(1-\delta)-\mu\delta\right).
    \]
\end{lemma}

\begin{proof}
For any \(t>0\), Markov's inequality yields
\[
\Prb\bigl[S \le (1-\delta)\mu\bigr] = \Prb\bigl[e^{-tS} \ge e^{-t(1-\delta)\mu}\bigr]
\le e^{t(1-\delta)\mu}\,\mathbb{E}\bigl[e^{-tS}\bigr].
\]
By independence and convexity of \(x\mapsto e^{-tx}\) on \([0,1]\),
\[
\mathbb{E}\bigl[e^{-tX_i}\bigr] \le 1-\mu_i + \mu_i e^{-t} \le \exp\bigl(\mu_i(e^{-t}-1)\bigr),
\]
where \(\mu_i = \mathbb{E}[X_i]\). Hence \(\mathbb{E}[e^{-tS}] \le \exp\bigl(\mu(e^{-t}-1)\bigr)\).
Thus
\[
\Prb\bigl[S \le (1-\delta)\mu\bigr] \le \exp\Bigl(\mu\bigl(t(1-\delta) + e^{-t} - 1\bigr)\Bigr).
\]
Choose \(t = -\log(1-\delta)>0\). Then \(e^{-t}=1-\delta\) and
\begin{align*}
    t(1-\delta) + e^{-t} - 1 
    = -(1-\delta)\log(1-\delta) + (1-\delta)-1 = -(1-\delta)\log(1-\delta) - \delta.
\end{align*}
Substituting yields the claimed bound.
\end{proof}

We next introduce a simple criterion for $F_{s,t}^{r}$-free $r$-graphs. 
\begin{lemma}\label{lem:complement}
Let \(2\le s,t\le r-1\) and \(s+t\ge r\). The $r$-graph $G$ is $F^r_{s,t}$-free if and only if for every pair of disjoint sets
$S,T\subseteq V(G)$ with $|S|=s$ and $|T|=t$, there is an edge $e\in \overline{G}$ such that
\[
S\subseteq e\subseteq S\cup T \quad\text{or}\quad T\subseteq e\subseteq S\cup T.
\]
\end{lemma}

\begin{proof}
The candidate copy of $F^r_{s,t}$ on $S\cup T$ consists exactly of the $r$-sets contained in
$S\cup T$ that contain $S$ or contain $T$. Hence $G$ contains this copy if and only if all these
candidate edges lie in $G$. Equivalently, $G$ avoids this copy if and only if at least one candidate
edge lies in the complement $\overline{G}$.
\end{proof}

Inspired by this lemma, we introduce the following ``shadow-complete'' concept.  

\begin{definition}
Let $n\ge r \ge s \ge 2$ and $R\geq 0$ be integers. An $r$-graph $G$ is 
{\it $(s,R)$-shadow-complete} if for every $(r+R)$-set $U\subseteq V(G)$ and every $s$-set
$S\subseteq U$, there is an edge $e\in G$ such that $S\subseteq e\subseteq U$.
Let
\[
\theta(n,r,s,R):=
\frac{1}{\tbinom{n}{r}}
\min\left\{|G|:G\subseteq\tbinom{[n]}{r}\text{ is }(s,R)\text{-shadow-complete}\right\},
\]
and define
\[
\theta(r,s,R):=\lim_{n\to\infty}\theta(n,r,s,R).
\]
\end{definition}
This limit exists by the usual averaging argument. Indeed, if \(m<n\), then taking
a uniformly random \(m\)-vertex induced subgraph of an optimal \(n\)-vertex
\((s,R)\)-shadow-complete \(r\)-graph shows that
\[
\theta(m,r,s,R)\le \theta(n,r,s,R).
\]
Thus \(\theta(n,r,s,R)\) is nondecreasing in \(n\), and in particular $\theta(n,r,s,R)\le \theta(r,s,R)$ for every \(n\ge r\).

We note that the above limit estimation provides an important upper bound on $q_{r,i}$ for our approach (but not an exact characterization). 

\begin{proposition}\label{prop:shadow-implies}
For $2\le i\le r-1$,
\[
q_{r,i}\le \theta(r,i,i-1).
\]
\end{proposition}

\begin{proof}
Let $G$ be $(i,i-1)$-shadow-complete. Given any pair of disjoint subsets $S,T\subseteq V(G)$ with $|S|=i$ and $|T|=r-1$, set
$U=S\cup T$. Then $|U|=r+i-1$. By shadow-completeness, there exists $e\in G$ with
$S\subseteq e\subseteq U$. By Lemma~\ref{lem:complement}, $\overline{G}$ is $F^r_{i,r-1}$-free. Taking asymptotic densities gives the claim.
\end{proof}


\section{Proof of Theorem~\ref{thm:small-i}}\label{sec:bound-fixed-i}
In this section, we prove Theorem~\ref{thm:small-i} by providing the desired upper bound on $q_{r,a+1}$ via $(a+1,a)$-shadow-complete $r$-graphs.
Our construction follows the recursive idea used by Pikhurko~\cite{Pikhurko} in his construction of Tur\'an systems.
The goal in \cite{Pikhurko} is to construct an \(r\)-graph such that every \((r+a)\)-set $U$ contains an $r$-edge. 
Very roughly speaking, it first uses a random hypergraph on the initial part of an ordered vertex set, and then treats the uncovered sets recursively on the remaining vertices.

We use a similar two-step scheme for the desired \((a+1,a)\)-shadow-complete $r$-graphs. The key difference is that, for each \((r+a)\)-set \(U\), we must cover every \((a+1)\)-subset of \(U\), rather than merely ensuring the presence of an edge in \(U\).
This conceptually will be carried out in the first step -- Lemma~\ref{lem:random-seed} (using general parameters $s,R$ instead of $a+1, a$, respectively), where we use the probabilistic method to cover all \((a+1)\)-subsets for most \(k\)-sets, where $k$ is chosen to be an intermediate parameter satisfying $a+1\leq k < r-a$.
These uncovered sets are then handled in the second step via a recursive shadow-complete construction on the remaining vertices.

\begin{lemma}\label{lem:random-seed}
Let $R,s,k$ be integers with $R\geq 1,k\ge s$, and let $h\ge s$.  Let $c$ satisfy $ h\le c\le \tbinom{k}{R}, $
and let $\Gamma$ be the random $R$-graph on a fixed $k$-set in which each $R$-edge is
selected independently with probability $p=c/\binom{k}{R}$.  Then
\begin{equation}\label{eq:random-seed-main}
\Prb\bigl[\alpha(\Gamma)\ge k-s\bigr]
\le
c^h h^{-h}e^{h-c}
+
\binom{k}{s}\left(\frac{sR}{k-R+1}\right)^{h+1}, 
\end{equation}
where \(\alpha(\Gamma)\) denotes the independence number of \(\Gamma\). 
\end{lemma}

\begin{proof}
For an $s$-set $Y$, let $A_Y$ be the event that every selected edge of $\Gamma$ meets $Y$.
The event $\alpha(\Gamma)\ge k-s$ is the union of the events $A_Y$ over all $Y\in\binom{[k]}{s}$.  
We proceed to estimate the probability of $A_Y$ conditioned on the event $e(\Gamma)=j$.
If $A_Y$ occurs, then all $j$ selected $R$-sets lie in the family of $R$-sets meeting $Y$.  If $k-s<R$, then the desired bound \eqref{eq:random-seed-main} is trivial since $sR/(k-R+1)\ge 1$ (if $R\ge 2$, this is evident; if $R=1$, then $k=s$ and thus this is equal). Otherwise,
\[
\frac{\tbinom{k-s}{R}}{\tbinom{k}{R}}
=\prod_{\ell=0}^{R-1}\frac{k-s-\ell}{k-\ell}
\ge \left(\frac{k-s-R+1}{k-R+1}\right)^R,
\]
and we have
\begin{align*}
\Prb\bigl[A_Y\mid e(\Gamma)=j\bigr]
\le \left(1-\frac{\tbinom{k-s}{R}}{\tbinom{k}{R}}\right)^j 
\le \left( 1-\left( \frac{k-s-R+1}{k-R+1}\right)^R\right)^j 
\le \left(\frac{sR}{k-R+1}\right)^j,
\end{align*}
where the last inequality holds by Bernoulli inequality $x^R\geq R(x-1)+1$ for $x\geq 0$ and $R\geq 1$.
If \(c=h\), then $c^h h^{-h}e^{h-c}=1$, and the desired bound is trivial. Hence we may assume \(c>h\).
Since $e(\Gamma)$ has mean $c$ and $h< c$, applying Lemma~\ref{LEM:Chernoff} with $\mu=c$ and $\delta=1-h/c\in (0,1)$ yields 
\[
\Prb\bigl[e(\Gamma)\le h\bigr]
\le \exp\left\{-c+h\log(ec/h)\right\}
= c^h h^{-h}e^{h-c}.
\]
On the event $j\ge h+1$, take the union bound over $Y$. If \(sR/(k-R+1)\ge 1\), then the second term on the right-hand side of~\eqref{eq:random-seed-main} is at least one, so the bound is trivial. Otherwise \(sR/(k-R+1)<1\), and for \(j\ge h+1\) the preceding bound is at most its value at \(j=h+1\). This gives 
\begin{align*}
    \Prb\bigl[\alpha(\Gamma)\ge k-s\bigr] &=\Prb\bigl[\cup_{Y} A_{Y}\bigr] \le \Prb\bigl[e(\Gamma)\le h\bigr]
     +\Prb\bigl[e(\Gamma)\ge h+1\bigr]\cdot
    \Prb\bigl[\cup_{Y}A_{Y}\mid e(\Gamma)\ge h+1\bigr] \\
    & \le c^h h^{-h}e^{h-c}
    + \sum_{Y}\Prb\bigl[A_{Y}\mid e(\Gamma)\ge h+1\bigr] \le c^h h^{-h}e^{h-c}
    +\binom{k}{s}\left(\frac{sR}{k-R+1}\right)^{h+1}, 
\end{align*}
which completes the proof. 
\end{proof}

We now proceed to prove a fixed-parameter upper bound for shadow-complete systems. Define
\[
\Lambda_a(r):=\theta(r,a+1,a)\tbinom{r-1}{a}.
\]
Our goal is to prove $\Lambda_a(r)=O_a(1)$, which would imply the upper bound for $q_{r,a+1}$ by Proposition~\ref{prop:shadow-implies}. 
In the next lemma, we will first use Lemma~\ref{lem:random-seed} to obtain an intermediate bound on $\Lambda_a(r)$.

\begin{lemma}\label{lem:recursion}
Let \(a\ge1\), \(r>a+1\), and \(a+1\le k\le r-a-1\). Let $D>1$ and
$m\in\mathbb N$ satisfy
\[
Dm(a+1)\le \tbinom{k}{a}.
\]
Set 
\[
p_0=(De^{1-D})^{m(a+1)}+\binom{k}{a+1}
\left(\frac{(a+1)a}{k-a+1}\right)^{m(a+1)+1}.
\]
Then
\begin{equation}\label{eq:recursion}
\Lambda_a(r)
\le
\binom{r-1}{a}
\left(
\frac{Dm(a+1)}{\tbinom{k}{a}}
+p_0\frac{\Lambda_a(r-k)}{\tbinom{r-k-1}{a}}
\right).
\end{equation}
\end{lemma}

\begin{proof}
Let $n$ be sufficiently large. 
We aim to construct an $(a+1,a)$-shadow-complete $r$-graph on $[n]$ whose edge density satisfies the claimed bound. Label the vertices $1,2,\dots,n$ in increasing order.  

Let $\mathcal S$ be a random $(k-a)$-graph on $[n]$ where each $(k-a)$-set appears independently with probability $p = Dm(a+1)/\binom{k}{a}$.  
Define the first $r$-graph $G_1$ by including an $r$-set $\{x_1<\dots<x_r\}$ if and only if its $k-a$ smallest vertices form an edge of $\mathcal S$.  
For a $k$-set $X=\{x_1<\dots<x_k\}$, we say $X$ is \emph{bad} if there exists an $(a+1)$-set $Z\subseteq X$ such that no member of $\mathcal S[X]$ contains $Z$. Equivalently, the $a$-graph $\Gamma_X:=\{X\setminus Q:Q\in\mathcal S[X]\}$ on $X$ has independence number at least $k-a-1$. By Lemma~\ref{lem:random-seed} (with $s=a+1$, $R=a$, $c=Dm(a+1)$, $h=m(a+1)$), the probability that a fixed $k$-set is bad is at most $p_0$.  

For each bad \(k\)-set \(X=\{x_1<\cdots<x_k\}\) with \(x_k\le n-r+k\), choose an
\((a+1,a)\)-shadow-complete \((r-k)\)-graph \(H_X\) on
\(\{x_k+1,\dots,n\}\) with
\[
|H_X|=\theta(n-x_k,r-k,a+1,a)\binom{n-x_k}{r-k}.
\]
Now let $G_2$ be the $r$-graph obtained by adding all $r$-sets $X\cup f$ with $f\in H_X$ for every bad $k$-set \(X\) with \(x_k\leq n-r+k\).
Finally, we set $G=G_1\cup G_2$. 
We claim that $G$ is the desired $r$-graph.

First, we verify that $G$ is $(a+1,a)$-shadow-complete. Take any $(r+a)$-set $U=\{u_1<\dots<u_{r+a}\}$ and any $(a+1)$-set $Z\subseteq U$. Let $X=\{u_1,\dots,u_k\}$ be the $k$ smallest vertices of $U$.  

If $X$ is not bad, then every $(a+1)$-subset of $X$ is contained in some edge of $\mathcal S[X]$. Choose an $(a+1)$-set $Y\subseteq X$ that contains $Z\cap X$ (possible since $|X|=k\ge a+1$). Pick $Q\in\mathcal S[X]$ with $Y\subseteq Q$, and consider $e=Q\cup(U\setminus X)$. Because $|Q|=k-a$ and $|U\setminus X|=r+a-k$, we have $|e|=r$. All elements of $Q$ are $\le x_k$ while those of $U\setminus X$ are $>x_k$, so the $k-a$ smallest vertices of $e$ are exactly the elements of $Q$. Hence $e\in G_1[U]$, and $Z\subseteq Y\cup(U\setminus X)\subseteq e$.  

If $X$ is bad, then $|U\setminus X|=r+a-k\ge 2a+1\ge a+1$ (because $k\le r-a-1$). Extend $Z\setminus X$ to an $(a+1)$-set $Y'\subseteq U\setminus X$, which is possible as $|U\setminus X|\ge a+1$. Since $U\setminus X$ has size $r+a-k = (r-k)+a$, the recursive graph $H_X$ (which is $(a+1,a)$-shadow-complete on its vertex set) provides an edge $f\in H_X$ with $Y'\subseteq f\subseteq U\setminus X$. Then $e=X\cup f$ belongs to $G_2[U]$ and contains $Z$ by construction. Thus $G$ is indeed $(a+1,a)$-shadow-complete.  

It remains to bound the edge density. The expected density of $G_1$ is exactly $p = Dm(a+1)/\binom{k}{a}$. For $G_2$, condition on $\mathcal S$ and sum over the maximum element $y=x_k$ of a $k$-set $X$: there are $\binom{y-1}{k-1}$ choices of $X$ with that maximum. Each such $X$, if bad, contributes a recursive graph on $n-y$ vertices. Hence
\[
\mathbb{E}[|G_2|] \le \sum_{y=k}^{n-r+k} \binom{y-1}{k-1}\, p_0\cdot\theta(n-y,r-k,a+1,a)\binom{n-y}{r-k}.
\]
The identity $\sum_{y=k}^{n-r+k}\binom{y-1}{k-1}\binom{n-y}{r-k}=\binom{n}{r}$ and $\theta(n-y,r-k,a+1,a) \le \theta(r-k,a+1,a)$ give
\[
\frac{\mathbb{E}[|G_2|]}{\tbinom{n}{r}} \le p_0\cdot\theta(r-k,a+1,a).
\]
Therefore, there is an instance of the random choice of $\mathcal S$ satisfying
\[
\frac{|G|}{\tbinom{n}{r}} \le \frac{Dm(a+1)}{\tbinom{k}{a}} + p_0\,\theta(r-k,a+1,a).
\]
Letting $n\to\infty$ and using $\Lambda_a(r)=\theta(r,a+1,a)\binom{r-1}{a}$ yields the desired inequality.
\end{proof}

Using this lemma, we can promptly derive the following estimate on $\Lambda_a(r)$.

\begin{proposition}\label{prop:fixed-R}
For every fixed integer $a\ge 1$, there is a constant $F_a>0$ such that 
\(\Lambda_a(r)\le F_a\) holds for all integers \(r>a\).
In particular, we have
\(\theta(r,a+1,a)\le F_a\binom{r-1}{a}^{-1}.\)
\end{proposition}

\begin{proof}
Fix $a\geq 1$. Choose $D>1$ and $m$ large enough so that
\[
(2(a+1))^a(De^{1-D})^{m(a+1)}\le \frac14.
\]
Increase $m$ if necessary so that $m(a+1)>a+1$. For all sufficiently large $r$, take
\(k=r-\left\lfloor\frac{r}{a+1}\right\rfloor.\)
Then $r-k\ge a+1$ and the hypotheses of Lemma~\ref{lem:recursion} hold. Moreover,
\[
\frac{\tbinom{r-1}{a}}{\tbinom{k}{a}}=O_a(1),
\quad
\frac{\tbinom{r-1}{a}}{\tbinom{r-k-1}{a}}=O_a((a+1)^a),
\]
and the second term in $p_0$ tends to $0$ as $r\to\infty$, because $a$ and $m$ are fixed and
$m(a+1)+1>a+1$. Therefore, after increasing the threshold for $r$, Lemma~\ref{lem:recursion} implies
\[
\Lambda_a(r)\le A_a+\frac12\Lambda_a\left(\left\lfloor\frac{r}{a+1}\right\rfloor\right)
\]
for a constant $A_a$ depending only on $a$.
Choose \(r_0\) large enough so that \(\lfloor r/(a+1)\rfloor>a\) for all \(r\ge r_0\) and the preceding inequality is
valid whenever \(r\ge r_0\).
Since there are
only finitely many $a<r<r_0$, put
\[
B_a:=\max_{a<r<r_0}\Lambda_a(r),
\]
with the convention that $B_a=0$ if the set is empty. 
By the preceding inequality, an easy inductive argument shows that
\[
\Lambda_a(r)\le 2A_a+B_a
\]
for all $r\ge r_0$ (and also for all $r>a$), and hence the proposition follows.
\end{proof}

\begin{proof}[\bf Proof of Theorem~\ref{thm:small-i}]
The result follows by combining Proposition~\ref{prop:shadow-implies} with
Proposition~\ref{prop:fixed-R}.
\end{proof}

\section{Proof of Theorem~\ref{thm:near-endpoint}}\label{sec:near-endpoint}

We now prove the lower bound for Theorem~\ref{thm:near-endpoint}. 
We need the following lemma, which can be derived from the clique-counting theorem of de Caen in the form presented in~\cite{BBBZ}.

\begin{lemma}[\cite{deCaen,BBBZ}]\label{lem:decaen}
Fix integers $b\ge 2$ and $t\ge b$. Let $K$ be a $b$-uniform hypergraph on $N$ vertices with missing-edge density $\alpha$, and set
\[
\widehat\alpha = 1 - (1-\alpha)\frac{N-b+1}{N} = \alpha + (1-\alpha)\frac{b-1}{N}.
\]
If $1-\binom{j-1}{b-1}\widehat{\alpha}\ge 0$ for every $b\le j\le t$, then the density of $t$-vertex cliques in $K$ satisfies
\[
\frac{m_t}{\tbinom{N}{t}} \ge \prod_{j=b}^{t} \left(1 - \binom{j-1}{b-1}\widehat\alpha\right).
\]
For fixed $b$, sufficiently large $t$, and $\alpha\binom{t-1}{b-1}=o_{t}(1)$, letting $N\to \infty$, then the right-hand side is asymptotically
\(\exp\left(-(1+o_{t}(1))\,\alpha\binom{t}{b}\right),\)
uniformly over such $\alpha$.
\end{lemma}

\begin{proof}
Let $m_j$ denote the number of $j$-vertex cliques in $K$. We have
\[
|E(K)| = (1-\alpha)\tbinom{N}{b} = (1-\widehat\alpha)\tbinom{N}{b}\frac{N}{N-b+1}.
\]
A direct application of Corollary~10 of Balogh, Bohman, Bollob\'as, and Zhao~\cite{BBBZ} with the parameter $\widehat\alpha$ would give the desired inequality
\[
\frac{m_t}{\tbinom{N}{t}} \ge \prod_{j=b}^{t} \left(1 - \binom{j-1}{b-1}\widehat\alpha\right).
\]

For the asymptotic statement, we first send $N\to\infty$ while keeping $b$ and $t$ fixed.  Then $\widehat\alpha \to \alpha$.  Under the assumption $\alpha\binom{t-1}{b-1}=o_{t}(1)$, all factors in the product are positive for sufficiently large $t$.  Moreover,
\[
\sum_{j=b}^{t} \binom{j-1}{b-1}\alpha = \alpha\tbinom{t}{b},
\]
and the sum of squares is
\[
\sum_{j=b}^{t} \left(\binom{j-1}{b-1}\alpha\right)^2 = O\left(\alpha\binom{t-1}{b-1}\right)\,\alpha\binom{t}{b} = o_{t}\left(\alpha\binom{t}{b}\right).
\]
Taking logarithms of the product and expanding $\log(1-x) = -x + O(x^2)$ yields
\[
\log\left(\prod_{j=b}^{t} \left(1 - \binom{j-1}{b-1}\alpha\right)\right) = -(1+o_{t}(1))\,\alpha\binom{t}{b},
\]
which exponentiates to the claimed asymptotic bound.  The uniformity over $\alpha$ follows because the error terms depend only on $t$ and $\alpha\binom{t-1}{b-1}=o_{t}(1)$.
\end{proof}

We are ready to prove Theorem~\ref{thm:near-endpoint}.

\begin{proof}[\bf Proof of Theorem~\ref{thm:near-endpoint}]
Let $G$ be an $n$-vertex $r$-graph with edge density $q=q(G)$ such that $\overline{G}$ is $F^r_{r-b,r-1}$-free. Throughout this proof, all \(o(1)\) terms are taken as \(n\to\infty\) first and then \(r\to\infty\), with \(b\) fixed. Fix an $(r-b)$-set $S\subseteq V(G)$ and consider its $b$-uniform link $L_G(S)$. Let $\alpha_S$ denote its edge density.  Averaging over all $S$ gives $\mathbb{E}_S[\alpha_S] = q$.

A set $T\in\binom{V(G)\setminus S}{r-1}$ is called an {\bf $S$-clique} if $\binom{T}{b}\cap L_G(S)=\emptyset$; equivalently, none of the $r$-edges $S\cup A$ with $A\in\binom{T}{b}$ belongs to $G$ (so they all lie in $\overline{G}$).  
If $T$ is an $S$-clique, then to avoid a copy of $F^r_{r-b,r-1}$ in $\overline{G}$ on the vertex set $S\cup T$, at least one of the $r-b$ edges
\(T\cup\left\{s\right\}\) for $s\in S$ must be present in $G$. Let \(\kappa(S)\) denote the proportion of \((r-1)\)-sets \(T\subseteq V(G)\setminus S\) that are \(S\)-cliques.

We first derive a lower bound on $\mathbb{E}[\kappa(S)]$.  If $qC$ is already larger than a sufficiently large constant multiple of $\log(C/(r-b))$, then the desired bound is immediate. Hence we may assume
\begin{equation}\label{eq:qC-assume}
qC = O\left(\log\frac{C}{r-b}\right) = O_b(\log r).
\end{equation}
Set $D := \binom{r-2}{b-1}$.  Then
\(qD = \frac{b}{r-1}\,qC = o(1).\)
Choose a parameter $\eta = \eta(r)$ such that
\[
\eta\to 0,\quad \frac{qD}{\eta}=o(1),\quad \mbox{and} \quad \frac{qD}{\eta}\,qC = o(1).
\]
For instance, under~\eqref{eq:qC-assume} one may take $\eta = (\log^2 r / r)^{1/3}$.
Call an $(r-b)$-set $S$ \emph{good} if $\alpha_S D \le \eta$.  By Markov's inequality and $\mathbb{E}_S[\alpha_S]=q$, the proportion $\mu$ of bad $S$ satisfies
\[
\mu \le \frac{qD}{\eta} = o(1),\quad\text{and consequently}\quad \mu qC = o(1).
\]
For every good $S$, we apply Lemma~\ref{lem:decaen} to the link $L_{\overline{G}}(S)$ (note that $\alpha_S$ is the edge density of $L_G(S)$, hence the missing-edge density of $L_{\overline{G}}(S)$ is also $\alpha_S$).  Since $\alpha_S D\le\eta=o(1)$, the asymptotic part of Lemma~\ref{lem:decaen} implies that $\kappa(S)$ is asymptotically at least
\[
\exp\bigl(-(1+o(1))\,\alpha_S C\bigr),
\]
uniformly over good $S$.
Applying Jensen's inequality over the good sets \(S\) gives
\[
\mathbb E[\kappa(S)]
\ge
(1-\mu)\exp\left(-(1+o(1))C\,
\mathbb E[\alpha_S\mid S\text{ good}]\right).
\]
Since $(1-\mu)\mathbb E[\alpha_S\mid S\text{ good}]\le \mathbb E[\alpha_S]=q$, and \(\mu qC=o(1)\), we obtain
\begin{equation}\label{eq:relative-clique-lower}
\mathbb{E}\left[\kappa(S)\right] \ge \exp\bigl(-(1+o(1))\,qC\bigr).
\end{equation}

We now obtain an upper bound on $\mathbb{E}\left[\kappa(S)\right]$ by counting ``witnesses''.  An edge $e\in G$ can witness a pair $(S,T)$ only if $e = T\cup\{s\}$ for some $s\in S$.  For a given $e$, there are at most $r$ choices for $s\in e$; then $T = e\setminus\{s\}$ is determined, and the remaining $r-b-1$ vertices of $S$ must lie outside $e$.  Hence the total number of ordered pairs $(S,T)$ witnessed by edges of $G$ is at most $r|G| \cdot \binom{n-r}{r-b-1}$.

The total number of ordered pairs \((S,T)\) with
\(|S|=r-b\), \(|T|=r-1\), and \(S\cap T=\emptyset\) is
\(\binom{n}{r-b}\binom{n-r+b}{r-1}\). Thus the witness density is at most
\[
\frac{r|G|\tbinom{n-r}{r-b-1}}
{\tbinom{n}{r-b}\tbinom{n-r+b}{r-1}}
=
(r-b)q+o(1).
\]
Every $S$-clique must be witnessed (otherwise $\overline{G}$ would contain a copy of $F^r_{r-b,r-1}$), so inequality~\eqref{eq:relative-clique-lower} implies
\[
\exp\bigl(-(1+o(1))\,qC\bigr) \le (r-b)q.
\]
Set $y = qC$.  Then the inequality becomes $y e^{(1+o(1))y} \ge C/(r-b)$.  Solving for $y$ in terms of the Lambert $W$-function gives
\[
q(G)\cdot C=qC=y \ge (1-o(1))\,W\!\left(\frac{C}{r-b}\right).
\]
Since \(G\) was arbitrary, taking the infimum over all such deletion graphs and then letting \(n\to\infty\) gives
\[
q_{r,r-b}\ge (1-o(1))\,\frac{W\!\left(C/(r-b)\right)}{C}.
\]
For fixed $b\ge2$, the asymptotic expansion $W(C/(r-b)) = (b-1)\log r + O_b(\log\log r)$ yields the simplified form
\[
q_{r,r-b} \ge (1-o(1))\,\frac{(b-1)\log r}{\tbinom{r-1}{b}}.
\]
The proof is complete. 
\end{proof}

\section{Concluding remarks}

In this note, we have sharpened the asymptotic estimates for the deficits $q_{r,i}$ of the hypergraphs $B_i^{(r)}$ near both ends of the considered sequence of hypergraphs.
More precisely, for every fixed $a\ge 1$, we have proved that $q_{r,a+1}=\Theta_a(r^{-a})$, closing the logarithmic gap left by the work of Balogh et al.~\cite{BBBZ}.  
At the opposite end, for every fixed codimension $b\ge 2$, we established $q_{r,r-b}=\Theta_b(r^{-b}\log r)$, matching the upper bound of Balogh et al.~\cite{BBBZ} through a new logarithmic lower bound obtained by clique-counting.

The only remaining case where the logarithmic factor is not yet resolved is the endpoint $i=r-1$ (equivalently $b=1$).  
It would be interesting to determine the precise asymptotic behavior of $q_{r,r-1}$. 

Another direction for future work is to investigate the behavior of $q_{r,i}$ when $i$ grows with $r$ in other regimes, for instance when $i = \Theta(r)$.  Our results indicate a sharp change in behavior: for fixed $i$ there is no logarithmic factor, while for $i = r-b$, with fixed $b\ge2$, a logarithmic factor appears. A natural problem is to determine the precise threshold where this transition occurs.  

\section*{Acknowledgments}
H.L. was supported by the National Natural Science Foundation of China (12501487), by the China Scholarship Council, and by the Institute for Basic Science (IBS-R029-C4). 
J.M. was supported by the National Key Research and Development Program of China (2023YFA1010201) and the National Natural Science Foundation of China grant 12125106.
T.W. and T.Z. were supported by the Innovation Program for Quantum Science and Technology (2021ZD0302902).

\section*{Declaration on the use of AI}
The authors used AI-assisted tools solely in the proof exploration and presentation of
Theorem~\ref{thm:near-endpoint}. 
All other mathematical ideas and proofs were developed independently by the authors without the use of AI-assisted tools. In particular, the proof of Theorem~\ref{thm:small-i} was inspired by the work of \cite{Pikhurko}, and the underlying approach dates back to November 2024. 
Each author has carefully checked all mathematical content of the paper and accepts full responsibility for its correctness.


\bibliographystyle{abbrv}

\end{document}